\documentclass[12pt,reqno,twoside]{amsart}
\usepackage{amsthm,amsmath,amssymb,mathrsfs,amsbsy,mathabx}
\usepackage{graphicx,epsfig,wrapfig,sidecap,subfig}
\usepackage[all]{xy}
\usepackage{proof}
\usepackage[usenames]{color}
\usepackage{hyperref}
\usepackage{bm}
\usepackage[capitalize]{cleveref}
\usepackage{listings}
\usepackage{color}

\usepackage[a4paper, hmargin=2.5cm, vmargin={3.2cm, 3.2cm}]{geometry}

\xyoption{dvips}
\xyoption{color}

\usepackage{color}

\title[Continuously many BD non-equivalences in substitution tiling spaces]{Continuously many bounded displacement non-equivalences in substitution tiling spaces}
\date{}
\author{Yaar Solomon}
\address{Ben-Gurion University of the Negev, Israel, {\tt yaars@bgu.ac.il}}

\newcommand{\N}{{\mathbb{N}}}
\newcommand{\Z}{{\mathbb{Z}}}

\newcommand{\R}{{\mathbb{R}}}
\newcommand{\C}{{\mathbb{C}}}

\newcommand{\X}{{\mathbb{X}}}
\newcommand{\Y}{{\mathbb{Y}}}

\newcommand{\PP}{\mathcal{P}}

\newcommand{\FF}{{\mathcal{F}}}

\newcommand{\QQ}{\mathcal{Q}}
\newcommand{\RR}{\mathcal{R}}
\newcommand{\TT}{\mathcal{T}}

\renewcommand{\SS}{\mathcal{S}}

\newcommand{\ee}{\mathbf{e}}
\newcommand{\pp}{\mathbf{p}}
\newcommand{\qq}{\mathbf{q}}
\newcommand{\vv}{\mathbf{v}}
\newcommand{\uu}{\mathbf{u}}

\newcommand{\xx}{\mathbf{x}}
\newcommand{\yy}{\mathbf{y}}

\newcommand{\PPP}{\mathscr{P}}

\newcommand{\bd}{\stackrel{\rm{\scriptscriptstyle BD}}{\sim}}

\newcommand{\cont}{2^{\aleph_0}}

\newcommand{\df}{{\, \stackrel{\mathrm{def}}{=}\, }}

\renewcommand{\span}[1]{\text{span}\!\left\{ {#1}\right\}}
\newcommand{\diam}[1]{\text{diam}\!\left({#1}\right)}
\newcommand{\supp}[1]{\text{supp}\!\left({#1}\right)}
\newcommand{\absolute}[1] {\left|{#1}\right|}

\newcommand{\norm}[1]{\left\|{#1}\right\|}
\newcommand{\inpro}[2]{\langle{#1},{#2}\rangle}

\newcommand{\vol}{\mathrm{vol}}
\newcommand{\BD}{\mathrm{BD}}

\newcommand{\ignore}[1]  {}

\theoremstyle{plain}
\newtheorem{thm}{Theorem}[section]

\newtheorem{lem}[thm]{Lemma}
\newtheorem{prop}[thm]{Proposition}
\newtheorem{cor}[thm]{Corollary}

\newtheorem{conj}[thm]{Conjecture}

\theoremstyle{definition}
\newtheorem{definition}[thm]{Definition}

\newtheorem{remark}[thm]{Remark}

\numberwithin{equation}{section}
\swapnumbers

\newif\ifdraft\drafttrue
\usepackage{color}

\begin{document}
\begin{abstract}
      We consider substitution tilings in $\R^d$ that give rise to point sets that are not bounded displacement (BD) equivalent to a lattice and study the cardinality of $\BD(\X)$, the set of distinct BD class representatives in the corresponding tiling space $\X$. We prove a sufficient condition under which the tiling space contains continuously many distinct BD classes and present such an example in the plane. In particular, we show here for the first time that this cardinality can be greater than one. 
\end{abstract}

\maketitle

\section{Introduction}\label{sec:introduction}
Let $X,Y\subset\R^d$ be two discrete sets, i.e. sets with no accumulation points. We say that $X$ is \emph{bounded displacement (BD) equivalent} to $Y$, and denote $X\bd Y$, if there exists a bijection $\phi:X\to Y$ that satisfies $\sup_{x\in X}\norm{x-\phi(x)}<\infty$. Such a mapping $\phi$ is called a \emph{BD-map}. In a similar manner we consider tilings of $\R^d$ by tiles of bounded diameter and inradius that is bounded away from zero. We say that such tilings $\SS$ and $\TT$ are \emph{BD-equivalent}, and denote $\SS\bd\TT$, if there are point sets $X_\SS$ and $Y_\TT$, which are obtained by placing a point in each tile of $\SS$ and $\TT$ respectively, so that $X_\SS$ and $Y_\TT$ are BD-equivalent. Note that since the tiles have bounded diameter, the question whether $\SS$ and $\TT$ are BD-equivalent or not does not depend on the choice of the point sets.

The BD-equivalence relation for general discrete point sets was studied in \cite{DO1,DO2,Laczk,DSS}, where the main focus was on point sets that are BD-equivalent to a lattice. We refer to such point sets as \emph{uniformly spread}, following Laczkovich, who gave an important criterion to check whether a point set is uniformly spread or not. We say that a tiling is \emph{uniformly spread} if its corresponding point set is such. Note that an application of the Hall's marriage theorem shows that every two lattices of the same co-volume are BD-equivalent, see \cite{DO2} or \cite{HKW} for a proof.

Recall that a point set $X\subset\R^d$ is a \emph{Delone set} if there exists $R,r>0$ so that $X$ intersects every ball of radius $R$ and every ball of radius $r$ contains at most one point of $X$. There are two fundamental families of Delone sets, which are on one hand non-periodic and on the other hand are often of finite local complexity and are repetitive, and hence are important objects of study in the theory of mathematical quasicrystals. One is the family of cut-and-project sets and the other is point sets that arise from substitution tilings, see \cite{BaakeGrimm} for further details. The question of which cut-and-project sets are uniformly spread was studied in \cite{HKW}, where in \cite{HK,HKK} the BD-equivalence of cut-and-project sets is linked to the notion of \emph{bounded remainder sets}. For substitution tilings (see \S \ref{sec:Background_and_definitions}), sufficient conditions to be uniformly spread were given in  \cite{ACG} and \cite{Solomon11}, which were further improved for tilings by tiles that are biLipschitz-homeomorphic to closed balls in \cite{Solomon14}, see Theorem \ref{thm:BD-criterion(Sol14)} below. In \cite{Smilansky-Solomon} multiscale substitution tilings are defined and a proof that they are never uniformly spread is given.

Questions regarding BD-equivalence and non-equivalence between two Delone sets, none of which is a lattice, were considered recently in \cite{FSS}. Using similar ideas to those of Laczkovich, a sufficient condition for BD-non-equivalence has been established. 

This paper studies the BD-equivalence relation on the tiling space $\X_\varrho$ of a fixed primitive substitution rule $\varrho$ in $\R^d$. In particular we are interested in the quantity $\absolute{\BD(\X_\varrho)}$, which is the cardinality of the quotient set $\X_\varrho/_{\bd}$. In particular, we show here for the first time that this cardinality can be greater than one. 

We denote by $\cont$ the cardinality of $\R$. Let $\varrho$ be a primitive substitution rule on the prototiles $\{T_1,\ldots,T_n\}$ in $\R^d$, with substitution matrix $M_\varrho$, whose eigenvalues are $\lambda_1>\absolute{\lambda_2}\ge\ldots\ge\absolute{\lambda_n}$, see \S \ref{sec:Background_and_definitions}. 
For a legal patch $P$
, we denote by $\vv(P)\in\R^n$ the vector whose $i$'th coordinate is the number of tiles of type $i$ in $P$. The notation $\bf{1}$ stands for the vector all of whose coordinates are equal to $1$, and $W^\perp$ denotes the orthogonal complement of a subspace $W$, with respect to the standard inner product in $\C^n$. When $dim(W)=1$ we denote by $\vv^\perp$ the orthogonal complement of $\span{\vv}$. 
Our main results are:


\begin{thm}\label{thm:main_result_general}
	Let $\varrho$ be a primitive substitution rule in $\R^d$ and let $t\ge 2$ be the minimal index for which $\lambda_t$ has an eigenvector whose sum of coordinates is non-zero. Assume that $\absolute{\lambda_t} > \lambda_1^{\frac{d-1}{d}}$ and that there exist two legal patches $P,Q$, such that 
	\begin{enumerate}
		\item 
		$\supp{P}$ and $\supp{Q}$ differ by a translation.\label{item:main_thm_1}
		\item
		$\vv(P)-\vv(Q)\notin \vv_t^\perp$, where $\vv_t$ is an eigenvector of $M_\varrho$ in $\bf{1}^\perp$, whose eigenvalue is equal in modulus to $\absolute{\lambda_t}$.    \label{item:main_thm_2}
	\end{enumerate} 
	Then $\absolute{\BD(\X_\varrho)} = \cont$. 
\end{thm}

Observe that every uniformly discrete set in $\R^d$, with separation constant $\delta>0$, is BD-equivalent to a subset of the lattice $\frac{\delta}{2}\Z^d$, hence the upper bound $\absolute{\BD(\X_\varrho)}\le\cont$ is trivial.  
We also remark that the primitivity assumption implies that the space $\X_\varrho$ is minimal with respect to the action of $\R^d$ by translations, see \cite{Q}. In particular, those continuously many tilings that we find, which are pairwise BD-non-equivalent, belong to the $\R^d$-orbit closure of any one of the tilings in space. 

\begin{cor}\label{cor:main_result_in_d=2}
	There exists a primitive substitution rule $\varrho$ on a set of two prototiles in the plane such that $\absolute{\BD(\X_\varrho)} = \cont$.	
\end{cor}

This corollary strengthen a similar result in the context of mixed substitution that was obtained in \cite{FSS}, that is where more than one substitution rule on the prototiles is allowed. 
In view of Theorem \ref{thm:BD-criterion(Sol14)} below, we also conjecture the following.

\begin{conj}\label{conj:1}
	Let $t\ge 2$ be as in Theorem \ref{thm:main_result_general}, and assume that $\absolute{\lambda_t} > \lambda_1^{\frac{d-1}{d}}$, then $\absolute{\BD(\X_\varrho)} = \cont$.
\end{conj}

\begin{remark}\label{remark:dichotomy}
	Conjecture \ref{conj:1} says that other than the case of equality, where $\absolute{\lambda_t} = \lambda_1^{\frac{d-1}{d}}$, and under certain regularity assumptions of the tiles, the following dichotomy holds: 
	\begin{itemize}
		\item
		$\absolute{\lambda_t} < \lambda_1^{\frac{d-1}{d}} \Longrightarrow \absolute{\BD(\X_\varrho)} = 1$.
		\item
		$\absolute{\lambda_t} > \lambda_1^{\frac{d-1}{d}} \Longrightarrow \absolute{\BD(\X_\varrho)} = \cont$.
	\end{itemize}
	In view of Theorem \ref{thm:BD-criterion(Sol14)} below, for tiles that are biLipschitz homeomorphic to closed balls, the former implication is clear, since in this case every tiling in $\X_\varrho$ is uniformly spread.  
	We also remark that in the case of equality both implications fail.  Indeed, as shown in \cite{FSS}, there are examples where equality holds and every $\TT\in\X_\varrho$ is uniformly spread and there are such examples where every $\TT\in\X_\varrho$ is not uniformly spread. In the latter, one can repeat the arguments of our Lemma \ref{lem:number_of_tiles_in_R^k-S^k} and Corollary \ref{cor:existence_of_two_non_BD_tilings} below, with the example in \cite{FSS}, showing that $\absolute{\BD(\X_\varrho)} > 1$ in a case of equality.
\end{remark}

\ignore{
\begin{remark}\label{remark:top_conj_not_imply_BD}
	Let $(\X,\R^d)$ and $(\Y,\R^d)$ be two topological spaces of patterns that are endowed with an $\R^d$ action. 
	Corollary \ref{cor:main_result_in_d=2} in particular implies that if $\X$ and $\Y$ are homeomorphic, or topologically conjugate, or MLD equivalent, or even equal, it is not enough to imply that every $X\in \X$ and $Y\in \Y$ are BD-equivalent. It is easy to construct examples that show that the implication in the other direction fails as well. Namely, examples where every $X\in \X$ and $Y\in \Y$ are BD-equivalent and $\X$ is not even homeomorphic to $\Y$. The notions that are mentioned in this remark can be found in \cite{BaakeGrimm}.
\end{remark}
}

The study of BD-equivalence is often linked with the \emph{bi-Lipschitz equivalence relation}, in which Delone sets are equivalent if there exists a bi-Lipschitz bijection between them. It is not hard to verify that BD-equivalence of Delone sets implies bi-Lipschitz equivalence. It was shown in \cite{Magazinov} that the cardinality of the set of Delone sets in $\R^d$ modulo bi-Lipschitz equivalence is $\cont$. Nonetheless, as all point sets that arise from primitive substitution tilings are bi-Lipschitz equivalent to a lattice (see \cite{Solomon11}), all the distinct BD-equivalence class representatives that we find here belong to the same bi-Lipschitz equivalence class.

\subsection*{Acknowledgments} 
The author thanks Jeremias Epperlein, Dirk Frettl\"oh, Scott Schmieding, Yotam Smilansky and Barak Weiss for useful discussions and remarks. 

\section{Background and definitions}\label{sec:Background_and_definitions}
We use bold figures to denote vectors in $\C^n$. The notation $\inpro{\cdot}{\cdot}$ stands for the standard inner product in $\C^n$, $\norm{\vv}:=\sqrt{\inpro{\vv}{\vv}}$, and $M^T$ (resp. $\uu^T$) denotes the transpose of a matrix $M$ (resp. vector $\uu$). This chapter contains preliminaries on tilings that are needed for our discussion. For further reading see \cite{BaakeGrimm}. 

A \emph{tile} $T$ is a compact subset of $\R^d$. A large variety of additional regularity assumptions on tiles appears in the literature. 
We assume here that $\mathcal{H}^{d-1}(\partial T)\in(0,\infty)$ for every tile $T$, where $\mathcal{H}^s(A)$ stands for the $s$-dimensional Hausdorff measure of the set $A\subset \R^d$, see \cite[Chap. 4]{Mattila}. 

A \emph{tiling} of a set $S\subset\R^d$ is a collection of tiles, with pairwise disjoint interiors, such that their union is equal to $S$. A tiling $P$ of a bounded set $B\subset\R^d$ is called a \emph{patch}, and we denote the set $B$, which is the \emph{support of $P$}, by $\supp{P}$. 
In particular, by our assumption on the tiles, for any scaling constant $t>0$ and any patch $P$ one has 
\begin{equation}\label{eq:Hausdorff_scaling}
\mathcal{H}^{d-1}(t\cdot\partial \supp{P}) = t^{d-1}\mathcal{H}^{d-1}(\partial \supp{P}) \in (0,\infty), 
\end{equation}
see \cite[p. 57]{Mattila}. This assumption is used in \eqref{eq:g_boundary_of_Am}.  

Tiles are called \emph{translation equivalent} if they differ by a translation and representatives of this equivalence relation are called \emph{prototiles}. The set of prototiles is denoted by $\FF$, and $\FF^*$ is the set of representatives of patches. Finally, given a tiling $\TT$ of $\R^d$, a bounded set $B\subset\R^d$ and a finite set $S$, we denote by $\# S := \text{ the cardinality of } S$ and by
\begin{equation}\label{eq:[B]^T_and_no.S}
[B]^\TT := \text{ the patch of } \TT \text{ that consists of all tiles that intersect } B.
\end{equation}

\subsection{Substitution tilings}\label{subsec:substitution_tilings}
Let $\xi>1$ and let $\mathcal{F}=\{T_1,\ldots,T_n\}$ be a set of tiles in $\R^d$. 

\begin{definition}\label{def:SubRule}
	A \emph{substitution rule} on $\FF$ is a fixed way to tile each one of the elements of $\xi\FF$ by the tiles in $\FF$. By applying $\varrho$ on a tile $T_i\in\FF$ we mean first scaling $T_i$ by $\xi$ and then substitute $\xi T_i$ by its fixed given tiling. Formally, it is a mapping $\varrho: \mathcal{F} \to \mathcal{F}^*$ satisfying $\supp{\xi T_i}=\supp{\varrho(T_i)}$ for every $i$. The number $\xi$ is the \emph{inflation factor} of $\varrho$.
\end{definition}

The function $\varrho$ can naturally be extended to $\FF^*$, and to tilings by tiles of $\FF$, by applying $\varrho$ separately to each tile. 

\begin{definition}\label{def:SubsTiling}
	Given a substitution rule $\varrho$ on $\mathcal{F}$ in $\R^d$, consider the patches: 
	\[\mathscr{L}_{\varrho}=\left\{\varrho^m(T) \: : \: m\in\N \: , \: T\in\mathcal{F} \right\}. \] 
	A patch is called \emph{legal} if it is a sub-patch of an element of $\mathscr{L}_\varrho$. The \emph{tiling space} $\X_{\varrho}$ is the collection of tilings of $\R^d$ with the property that every patch in them is legal. A tiling $\TT\in \X_{\varrho}$ is called a \emph{substitution tiling} that corresponds to $\varrho$. 
\end{definition}


\begin{definition}\label{def:SubMatrix+Primitive}
	The \emph{substitution matrix} $M_{\varrho}=(a_{ij})$ of $\varrho$ is defined by 
	\[a_{ij} = \#\left\{\text{tiles of type } i \text{ in } \varrho(T_j) \right\}.\] 
	$\varrho$ is called \emph{primitive} if $M_{\varrho}$ is a primitive matrix. Namely, if there exists an $m\in \N$ such that all entries of $M_{\varrho}^m$ are positive. 
\end{definition}   

\subsection{Further notations and properties} \label{subsec:Notations}

We assume throughout that $\varrho$ is primitive. Perron-Frobenius theorem then implies that the eigenvalues of $M_{\varrho}$ can be ordered such that  $\lambda_1>\absolute{\lambda_2}\ge \ldots \ge \absolute{\lambda_n}$. We denote by $(\vv_1,\ldots,\vv_n)$ a Jordan basis of $M_\varrho$, $\vv_i$ corresponds to $\lambda_i$. 

Given a patch $P$ in a tiling $\TT \in\X_{\varrho}$, let $\vv(P)\in\Z^n$ denote the vector whose $i$'th coordinate is the number of tiles of type $i$ in $P$. Denote by $\ee_i$ the $i$'th element of the standard basis of $\R^n$, so $e_i=\vv(T_i)$ and for every $m\in\N$ one has $M_{\varrho}^m \ee_i = \vv\left(\varrho^m(T_i)\right)$. 

Observe that $\lambda_1=\xi^d$, and that $M_{\varrho}^T\uu_1=\lambda_1\uu_1$, where $\uu_1$ is the vector whose $i$'th coordinate is $\vol(T_i)$. This implies that for every patch $P$ we have 
\begin{equation}\label{eq:inner_product_for_volume+counting}
\vol(\supp{P}) = \inpro{\uu_1}{\vv(P)} \qquad\text{ and }\qquad \# P = \inpro{\bf{1}}{\vv(P)},
\end{equation}
where ${\bf 1}:=(1,1,\ldots,1)^T$ and $\# F$ denotes the cardinality of a finite set $F$. 
  
\ignore{
	\begin{prop}\label{prop:TT^{(m)}_Def}\cite{Robinson}
		If $\varrho$ is a primitive substitution rule in $\R^d$, then $X_{\varrho}\neq\emptyset$ and for every $\TT\in X_{\varrho}$ and $m\in\N$ there exists a tiling $\TT^{(m)}\in X_{\varrho^{(m)}}$ that satisfies $\varrho^m(\TT^{(m)})=\TT$.   
	\end{prop} 
}

\subsection{Bounded displacement equivalence}\label{subsec:BD}
Given a tiling $\TT$ of $\R^d$ by tiles of bounded diameter, let $\Lambda_\TT\subset\R^d$ be a point set with a point in each tile of $\TT$ (taken with multiplicity in the case that a point $x\in T_1\cap T_2$ was chosen for tiles $T_1, T_2$ in $\TT$).
$\TT$ is called \emph{uniformly spread} if there is a BD-map $\phi: \Lambda_\TT \to \alpha\Z^d$, for some $\alpha>0$. Namely, a map that satisfies $\sup_{x\in\Lambda_\TT}\norm{x-\phi(x)}<\infty$.  
\begin{thm}\cite[Theorem 1.2]{Solomon14}\label{thm:BD-criterion(Sol14)}
	Let $\varrho$ be a primitive substitution rule on prototiles that are biLipschitz-homeomorphic to closed balls in $\R^d$, and let $\TT\in\X_\varrho$. 
	Let $t\ge 2$ be the minimal index for which the eigenvalue $\lambda_t$ has an eigenvector $\vv_t\notin\bf{1}^\perp$. 
	\begin{itemize}		
		\item[(I)]
		If $\absolute{\lambda_t}<\lambda_1^{\frac{d-1}{d}}$ then $\TT$ is uniformly spread. 
		\item[(II)]
		If $\absolute{\lambda_t}>\lambda_1^{\frac{d-1}{d}}$ then $\TT$ is not uniformly spread.
	\end{itemize} 
\end{thm}

Given $x=(x_1,\ldots,x_d) \in \R^d$, following \cite{FSS}, we denote by $C(x)$ the axis-parallel unit cube $\bigtimes_{i=1}^d [x_i-\frac{1}{2},x_i+\frac{1}{2})$ centered at $x$. Denote by $\QQ_d$ the set $\left\{C(x) \mid x\in \Z^d \right\}$ of lattice centered unit cubes, and let $\QQ_d^*$ be the collection of all finite subsets of $\QQ_d$. 

\begin{thm}\cite[Theorem 1.1]{FSS}\label{thm:from_Dirk+Yotam}
	Let $\Lambda_1, \Lambda_2$ be two Delone sets in $\R^d$ and suppose that there is a sequence $(A_m)_{m\in\N}$ of sets $A_m\in \QQ_d^*$ for which
	\begin{equation}\label{eq:non_BD_condition}
	\lim_{m\to\infty} \frac{| \#(\Lambda_1 \cap A_m) - \#(\Lambda_2 \cap A_m) |}{\mu_{d-1}(\partial A_m)} = \infty.
	\end{equation}   
	Then there is no BD-map $\phi: \Lambda_1 \to \Lambda_2$.
\end{thm}

\begin{remark}
	Theorem \ref{thm:from_Dirk+Yotam} originally includes the additional assumption that the Delone sets have \emph{box diameter $\ge 1$}. Namely that for every $x\in\R^d$ the cube $C(x)$ contains at most one element of each of the Delone sets. As also mentioned in \cite{FSS}, this additional assumption is unnecessary since one may replace the sets $\Lambda_i$ and $A_m$ by a mutual rescaling of them, with a suitable constant, so that this assumption holds. 
\end{remark}

Theorem \ref{thm:from_Dirk+Yotam} can also be stated in the language of tilings, where this new formulation follows directly from Theorem \ref{thm:from_Dirk+Yotam}, using the notation in \eqref{eq:[B]^T_and_no.S}.  We say that two given tilings are \emph{BD-non-equivalent} if there is no BD-map between their corresponding Delone sets. 

\begin{cor}\label{cor:from_Dirk+Yotam_for_tilings}
	Let $\TT_1, \TT_2$ be two tilings of $\R^d$. Suppose that there is a sequence of sets $A_m\in \QQ_d^*$ for which
	\begin{equation}\label{eq:non_BD_condition_tilings}
	\lim_{m\to\infty} \frac{\left| \#[A_m]^{\TT_1} - \#[A_m]^{\TT_2} \right|}{\mu_{d-1}(\partial A_m)} = \infty.
	\end{equation}   
	Then the tilings $\TT_1$ and $\TT_2$ are BD-non-equivalent.
\end{cor}

\section{Proof of Corollary \ref{cor:main_result_in_d=2}}\label{sec:Example}
We begin with an example proving Corollary \ref{cor:main_result_in_d=2}, relying on Theorem \ref{thm:main_result_general}. Consider the following substitution rule $\varrho$ on a set of two tiles $T_1$ and $T_2$ in the plane, where $T_1$ is a $1\times 1$ square and $T_2$ is a $2\times 1$ rectangle: 

\begin{figure}[ht!]
	\includegraphics[scale=0.5]{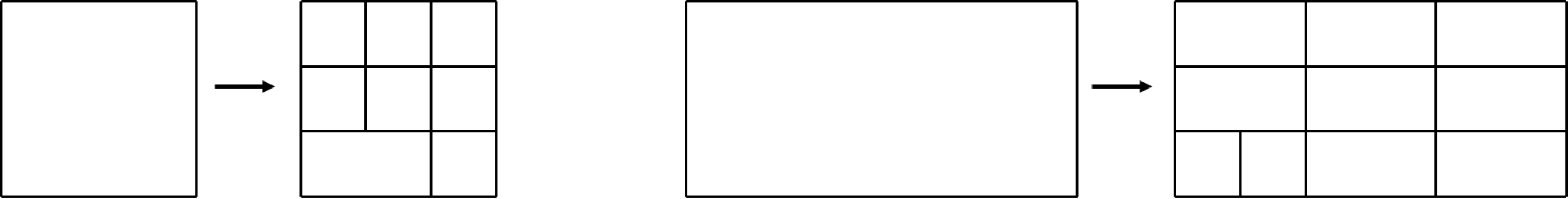}\caption{The substitution rule $\varrho$ on a square and rectangle.}\label{fig:subs_rule}
\end{figure}
Note that the corresponding substitution matrix here is $M_\varrho=\binom{7 \ 2}{1 \ 8}$, whose eigenvalues are $\lambda_1=9, \lambda_2=6$ and eigenvectors are $\vv_1=\binom{1}{1}, \vv_2=\binom{-2}{1}$ respectively.

For $k\in\N$ we denote by $R^{(k)}$ a translated copy of the patch 
 $\varrho^k(T_2)$ and by $S^{(k)}$ a translated copy of the patch supported on two adjacent patches of the form $\varrho^k(T_1)$, so that $\supp{S^{(k)}}$ is equal to $\supp{R^{(k)}}$, up to a translation. To indicate that these patches are centered at the origin we use the notations $R^{(k)}_0$ and $S^{(k)}_0$. 

\begin{lem}\label{lem:number_of_tiles_in_R^k-S^k}
	For every $k\in\N$ we have $\absolute{\#R^{(k)}-\#S^{(k)}}=6^k$.	
\end{lem}

\begin{proof}
	By the definition of the substitution matrix, the number of tiles in the patch $R^{(k)}$ (resp. $S^{(k)}$) is given by the sum of the coordinates of the vector $M^k_\varrho\binom{0}{1}$ (resp. $2M^k_\varrho\binom{1}{0}$). So the required quantity is the sum of the coordinates of the vector $M_\varrho^k \binom{0}{1}-2M_\varrho^k \binom{1}{0}=M_\varrho^k \binom{-2}{1}$. Since $\binom{-2}{1}=\vv_2$ we obtain 
	$M_\varrho^k \binom{-2}{1}=6^k\binom{-2}{1},$
	and hence
	\[\absolute{\#R^{(k)}-\#S^{(k)}}=\absolute{6^k\inpro{\binom{-2}{1}}{\binom{1}{1}}}=6^k.\]
\end{proof}

Observe that $R^{(2)}_0 = \varrho(R^{(1)}_0)$ (resp. $S^{(2)}_0$) contains a copy of a centered $R^{(1)}_0$ (resp. $S^{(1)}_0$). Thus the sequences $\left(R^{(k)}_0\right)_{k\in\N}$ and $\left(S^{(k)}_0\right)_{k\in\N}$ are nested sequences that define two fixed points of $\varrho$ in $\X_\varrho$ by 
\[\RR:=\bigcup_{k\in\N}R^{(k)}_0 \qquad \SS:=\bigcup_{k\in\N}S^{(k)}_0 .\]
Since $\mu_1(\partial\supp{R^{(k)}}) = \mu_1(\partial\supp{S^{(k)}}) = 2\cdot 3^k$, Corollary \ref{cor:existence_of_two_non_BD_tilings} below follows from Corollary \ref{cor:from_Dirk+Yotam_for_tilings} with $A_m:=\supp{R^{(m)}_0}$, and from Lemma \ref{lem:number_of_tiles_in_R^k-S^k}.

\begin{cor}\label{cor:existence_of_two_non_BD_tilings}
	The tilings $\RR$ and $\SS$ are BD-non-equivalent. 	
\end{cor}

The substitution rule in Figure \ref{fig:subs_rule} with the patches $R^{(1)}$ and $S^{(1)}$ also provide a proof for Corollary \ref{cor:main_result_in_d=2}. 

\begin{proof}[Proof of Corollary \ref{cor:main_result_in_d=2}]
	The substitution $\varrho$ in Figure \ref{fig:subs_rule} is primitive and we have $d=t=2$, $\absolute{\lambda_t}=6>3=\lambda_1^{(d-1)/d}$. The patches $P=R^{(1)}$ and $Q=S^{(1)}$ clearly satisfy the assumptions of Theorem \ref{thm:main_result_general} and hence $\absolute{\BD(\X_\varrho)}=\cont$. 	
\end{proof}

\section{Proof of Theorem \ref{thm:main_result_general}}\label{sec:Proof_of_main_general_thm}
This chapter contains the proof of Theorem \ref{thm:main_result_general}. 
Throughout this chapter, $\varrho$ is a primitive substitution rule defined on the set of prototiles $\FF=\{T_1,\ldots,T_n\}$ in $\R^d$, $\lambda_1>\absolute{\lambda_2}\ge\ldots\ge\absolute{\lambda_n}$ are the eigenvalues of $M_\varrho$, $(\vv_1,\ldots,\vv_n)$ is a corresponding Jordan basis, and $t\ge 2$ is as in Theorem \ref{thm:BD-criterion(Sol14)}.

\begin{lem}\label{lem:number_of_tiles_in_large_P-Q}
	Suppose that $P,Q$ are two legal patches of $\varrho$ and assume that 
	\begin{itemize}
		\item 
		$\vol(P)=\vol(Q)$.
		\item
		$\pp-\qq\notin \vv_t^\perp$, where $\pp=\vv(P), \qq=\vv(Q)$.
	\end{itemize} 
	Then there exist a constant $c_0 >0$ that depend on $P, Q$ and $\varrho$ such that
	\begin{equation}\label{eq:same_vol=>good_estimate}
	\absolute{\#\varrho^k(P) - \#\varrho^k(Q)} \ge c_0\absolute{\lambda_t}^k.	
	\end{equation}
\end{lem}

\begin{proof}
	Recall that $\uu_1$ denotes the first eigenvector of $M_\varrho^T$, thus $\uu_1^\perp = \span{\vv_2,\ldots,\vv_n}$. Since $\uu_1$ can be taken to be the vector of volumes of the prototiles, as in \eqref{eq:inner_product_for_volume+counting},
	\[\inpro{\uu_1}{\pp}=\vol(P)=\vol(Q)=\inpro{\uu_1}{\qq},\]
	and thus 
	
	\begin{equation}\label{eq:p-q_in_span{v_2,...,v_n}}
		\pp-\qq\in \uu_1^\perp = \span{\vv_2,\ldots,\vv_n}. 
	\end{equation}
	In addition, for every $k\in\N$ we have 
	
	\begin{equation}\label{eq:tiles_in_P-in_Q--1}
	\#\varrho^k(P) - \#\varrho^k(Q) = \inpro{\bf{1}}{M_\varrho^k(\pp)} - \inpro{\bf{1}}{M_\varrho^k(\qq)} = \inpro{\bf{1}}{M_\varrho^k(\pp-\qq)}.
	\end{equation}
	By \eqref{eq:p-q_in_span{v_2,...,v_n}},   
	\[M_\varrho^k(\pp - \qq) = \alpha_2\lambda_2^k\vv_2 + \ldots + \alpha_n\lambda_n^k\vv_n, \] 
	for some constants $\alpha_2,\ldots,\alpha_n\in\C$. But by the definition of $t$, for any $j<t$ we have $\inpro{\bf{1}}{\vv_j}=0$ and thus 
	
	\begin{equation}\label{eq:tiles_in_P-in_Q--2}
		\inpro{\bf{1}}{M_\varrho^k(\pp - \qq)} = \sum_{j=t}^n\inpro{\bf{1}}{\alpha_j\lambda_j^k\vv_j} = \sum_{j=t}^n\alpha_j\lambda_j^k\inpro{\bf{1}}{\vv_j}.
	\end{equation}
	Note that by assumption $\pp - \qq\notin (\span{\vv_t})^\perp$, then $\alpha_t\neq 0$. Combining \eqref{eq:tiles_in_P-in_Q--1} and \eqref{eq:tiles_in_P-in_Q--2} we see that 
	
	\[\absolute{\#\varrho^k(P) - \#\varrho^k(Q)} = \absolute{\sum_{j=t}^n\alpha_j\lambda_j^k\inpro{\bf{1}}{\vv_j}},\]
	and since $\alpha_t\neq 0$, the assertion follows.  
\end{proof}

Let $P$ and $Q$ be two legal patches and write $P=\varrho^{a_1}(T_i)$ and $Q=\varrho^{a_2}(T_j)$ with $a_1,a_2\in\N$ and $i,j\in\{1,\ldots,n\}$. For a patch $\PPP$ and a point $\xx\in\supp{\PPP}$ we use the notation 

\begin{equation}\label{eq:P_x}
\PPP_\xx := \text{ the translated copy of } \PPP \text{ in which } \xx \text{ is at the origin}.
\end{equation}
The primitivity of $\varrho$ is used for the simple observation that is given in the following lemma.

\begin{lem}\label{lem:xx(P)}
	There exists an $a_0\in\N$ such that 
	\begin{enumerate}
		\item
		The patch $\varrho^{a_0}(P)$ contains a patch $\PPP$, which is a translated copy of $P$ whose support is disjoint from the boundary of the support of $\varrho^{a_0}(P)$. 
		In particular, there is a (unique) point $\xx(P)\in\supp{P}$ so that the copy $\PPP$ in $\varrho^{a_0}(P_{\xx(P)})$ coincides with the patch $P_{\xx(P)}$. \label{item:prop_of_xx(P)}
		\item
		The patch $\varrho^{a_0}(P)$ also contains a translated copy of $Q$.\label{item:varrho^a(P)_contains_Q}	
	\end{enumerate}	
\end{lem}

\begin{proof}
By the primitivity of $\varrho$, for a large integer $b$, $\varrho^b(P)$ contains copies of all tile types, and also tiles of all types that are disjoint from $\partial \supp{\varrho^b(P)}$. Hence there exists some $a_0$ so that for every $a\ge a_0$ the patch $\varrho^{a}(P)$ contains translated copies of both $P$ and $Q$, which are disjoint from $\partial\supp{\varrho^{a}(P)}$. 
Fix a copy of $P$ in $\varrho^{a_0}(P)$, whose support is disjoint from $\partial \supp{\varrho^{a_0}(P)}$, and denote it by $\PPP$. 

The point $\xx(P)$ can be defined as follows. Repeating the above argument one finds a patch $\PPP_1$ inside $\varrho^{a_0}(\PPP)$, a patch $\PPP_2$ inside $\varrho^{a_0}(\PPP_1)$, etc. Each $\PPP_{m+1}$ is a copy of $P$ that sits inside $\varrho^{a_0}(\PPP_{m})$, thus the nested intersection $\bigcap_{m\in\N}\xi^{-ma_0}\PPP_m$ is a point that satisfies the requirements. 
\end{proof}

\ignore{	
\begin{figure}[ht!]
	\includegraphics[scale=0.5]{centered_copy}\caption{An illustration for the location of $\xx(P)$ in the copy of $P$ inside $\varrho^a(P)$ (resp. $Q$).}\label{fig:centered_copies}
\end{figure}
}

To prove Theorem \ref{thm:main_result_general} we explicitly construct continuously many distinct tilings, where each one of them is defined as an increasing union of a certain nested sequence of patches. 
To define these patches, we set the following notations.  

Let $P$ and $Q$ be two legal patches whose supports differ by a translation. We fix marked points $\xx(P)\in \supp{P}$ and  $\xx(Q)\in \supp{Q}$ as in Lemma \ref{lem:xx(P)}. For any scaled copy $\beta P$ of $P$ (resp. Q) we set $\xx(\beta P):=\beta\cdot\xx(P)$. We also fix the number $a$ to be the maximum between the values of $a_0$ that are obtained when applying Lemma \ref{lem:xx(P)} with $P$ and with $Q$. Then the patch $\varrho^a(P)$ contains a copy of $P$ and the patch $\varrho^a(Q)$ contains a copy of $Q$, as in Lemma \ref{lem:xx(P)}. We refer to these particular patches as
\begin{itemize}
	\item 
	\emph{the centered copy of $P$ in $\varrho^{a}(P)$}.
	\item 
	\emph{the centered copy of $Q$ in $\varrho^{a}(Q)$}.
\end{itemize}
Lemma \ref{lem:xx(P)} can be applied repeatedly. For integers $k<m$, the notions of  
\begin{itemize}
	\item 
	\emph{the centered copy of $\varrho^{ka}(P)$ in $\varrho^{ma}(P)$}
	\item 
	\emph{the centered copy of $\varrho^{ka}(Q)$ in $\varrho^{ma}(Q)$},
\end{itemize}
play an important role in the proof of Proposition \ref{prop:g_main_step} below, which is the core of the proof of Theorem \ref{thm:main_result_general}. 
We use the notation $\sigma^i\in\{P,Q\}^i$ for a finite sequence of length $i$, where $\sigma^i(\ell)$ denotes the $\ell$'th letter and $\sigma^i[1\ldots \ell]$ is the prefix of length $\ell$ of $\sigma^i$. 
Finally, relying on the assumption $\absolute{\lambda_t} > \lambda_1^{\frac{d-1}{d}}$ of Theorem \ref{thm:main_result_general}, we fix $h \in a\cdot\N$ to be the smallest multiple of $a$ that satisfies 
\begin{equation}\label{eq:h_and_k_i}
\lambda_1^{\frac1{h}} < \frac{\absolute{\lambda_t}}{\lambda_1^{(d-1)/d}} \ \, \quad\text{ and set }\quad k_i:= h^{i-1}.
\end{equation}  

\begin{prop}\label{prop:g_main_step}
	For every $i\in\N$ and every sequence $\sigma^i\in\{P,Q\}^i$ of length $i$ there exists a legal patch $\PPP_{\sigma^i}^{(k_i)}$ such that for $i=1$ we have $\PPP_{P}^{(1)}=P_{\xx(P)}$, $\PPP_{Q}^{(1)} = Q_{\xx(Q)}$, and so that the following properties hold for every $i\in\N$ and every $\sigma^i\in\{P,Q\}^i$:
	\begin{enumerate}
		\item 
		$\PPP_{\sigma^i}^{(k_i)}$ is a translated copy of $\begin{cases}
		\varrho^{k_i}(P), \text{ if } \sigma^i(i)=P \\
		\varrho^{k_i}(Q), \text{ if } \sigma^i(i)=Q
		\end{cases}$. \label{item:g_translated_copy_of..}
		\item
		If $\sigma^{i}$ is a prefix of $\sigma^{i+1}$ then $\PPP_{\sigma^{i+1}}^{(k_{i+1})}$ contains a copy of $\PPP_{\sigma^i}^{(k_i)}$ as a sub-patch, whose support contains the origin and is disjoint from the boundary of  $\supp{\PPP_{\sigma^{i+1}}^{(k_{i+1})}}$.\label{item:g_nested+contain_0} 
		\item
		$\norm{\xx\left(\PPP_{\sigma^{i+1}}^{(k_{i+1})}\right)}\le c_1\cdot \lambda_1^{k_i/d}$, where $c_1 = \lambda_1^{a/d}\diam{\supp{P}}$.\label{item:g_distance_of_center_from_0}
	\end{enumerate} 
\end{prop} 

\begin{proof}
	The proof is by induction on $i$. For $i=1$ we define $\PPP_{P}^{(1)}=P_{\xx(P)}$, $\PPP_{Q}^{(1)} = Q_{\xx(Q)}$.
	Suppose that the patches $\PPP_{\sigma^i}^{(k_i)}$ were defined and that the above properties hold for every $\sigma^i\in\{P,Q\}^i$, we define the patches $\PPP_{\sigma^{i+1}}^{(k_{i+1})}$ as follows. Fix some  $\sigma^{i+1}\in\{P,Q\}^{i+1}$. 
		
	Recall that $h$ is a multiple of $a$ and observe that $k_{i+1} = h^i$ is a much larger integer than $k_i+a=h^{i-1}+a$. If the $i+1$ letter of $\sigma^{i+1}$ is $P$, denote by $\mathscr{T}_P$ the centered copy of $\varrho^{(k_i+a)}(P)$ inside $\varrho^{(k_{i+1})}(P)$ (respectively, if $\sigma^{i+1}(i+1)=Q$ let $\mathscr{T}_Q$ be the centered copy of $\varrho^{(k_i+a)}(Q)$ inside $\varrho^{(k_{i+1})}(Q)$). 
	A key observation is that positioning $\mathscr{T}_P$ (resp. $\mathscr{T}_Q$) in $\R^d$ forces the position of the much larger patch $\varrho^{(k_{i+1})}(P)$ (resp. $\varrho^{(k_{i+1})}(Q)$) that contains it. Consider the centered copy of $\varrho^{(k_i)}(P)$ or the copy of $\varrho^{(k_i)}(Q)$ inside $\mathscr{T}_P$, which exists by Lemma \ref{lem:xx(P)}, depending on whether the $i$'th letter of $\sigma^{i+1}$ is $P$ or $Q$ (resp. inside $\mathscr{T}_Q$ consider the centered copy of $\varrho^{(k_i)}(Q)$ or the copy of $\varrho^{(k_i)}(P)$). 
	One of these two patches, depending on the $i$'th letter of $\sigma^{i+1}$, is a translated copy of the patch $\PPP_{\sigma^{i+1}[1\ldots i]}^{(k_i)}$ that we have obtained from the induction hypothesis. 
	{\bf We place $\mathscr{T}_P$ (resp. $\mathscr{T}_Q$) so that the above particular copy of $\varrho^{(k_i)}(P)$ or of $\varrho^{(k_i)}(Q)$ in it coincide with $\PPP_{\sigma^{i+1}[1\ldots i]}^{(k_i)}$}, see Figure \ref{fig:Positioning}. 
	The above placement fixes the position of the copy of the patch $\varrho^{(k_{i+1})}(P)$ or $\varrho^{(k_{i+1})}(Q)$ from which we have started, and we define this fixed patch to be  $\PPP_{\sigma^{i+1}}^{(k_{i+1})}$.  
	\begin{figure}[ht!]
		\includegraphics[scale=0.5]{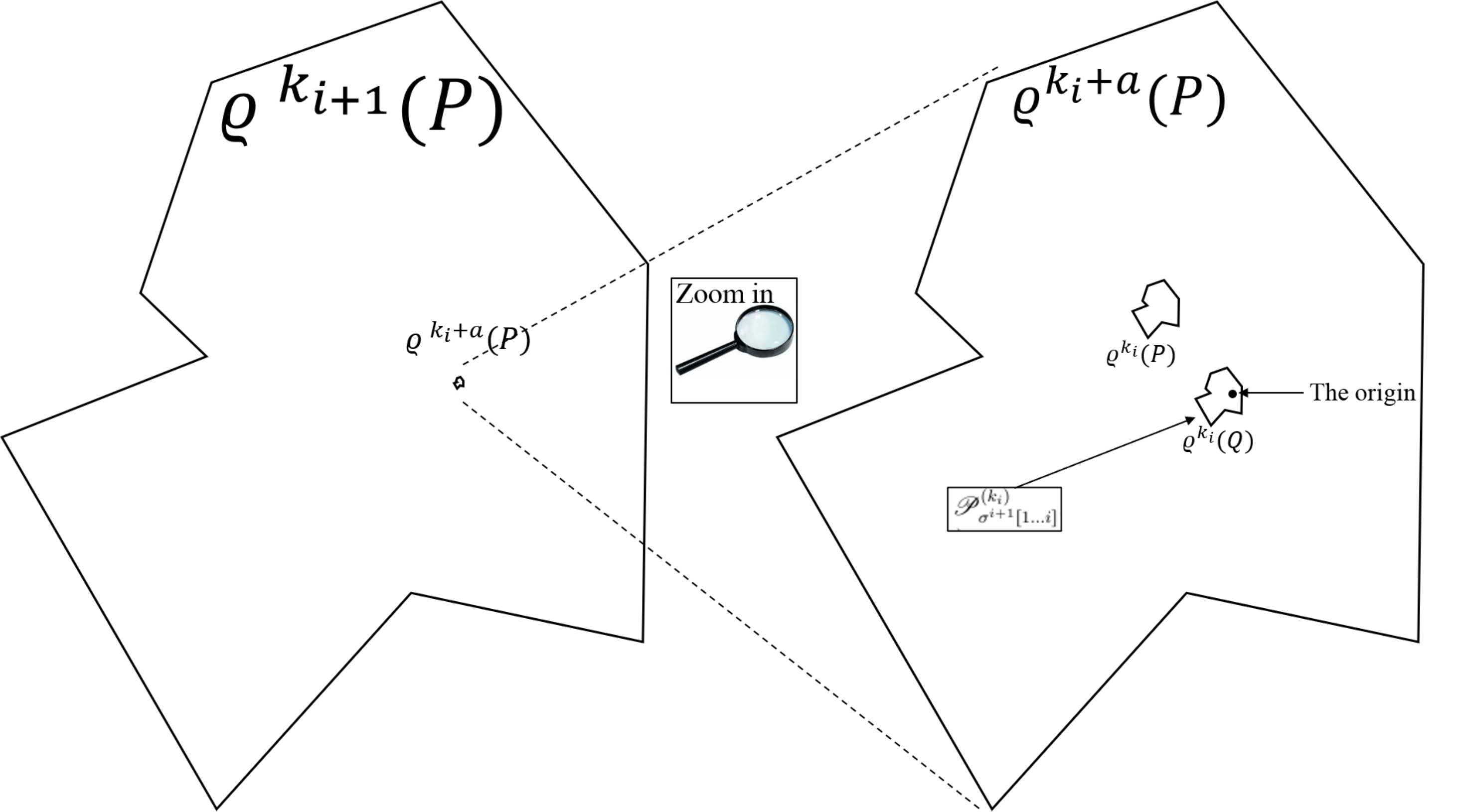}\caption{This picture corresponds to the case where the last two letters of $\sigma^{i+1}$ are $QP$ and it is done similarly for the other three possible options. The illustration shows how to position the patch $\varrho^{(k_{i+1})}(P)$, which is later defined to be $\PPP_{\sigma^{i+1}}^{(k_{i+1})}$, providing that we know the position of $\PPP_{\sigma^{i+1}[1\ldots i]}^{(k_i)}$, which was given to us by the induction hypothesis. In this picture, as the $i$'th letter of $\sigma^{i+1}$ is $Q$, the patch $\PPP_{\sigma^{i+1}[1\ldots i]}^{(k_i)}$ is a translated copy of $\varrho^{(k_i)}(Q)$. We place $\varrho^{(k_{i+1})}(P)$ such that the copy of $\varrho^{(k_i)}(Q)$ inside the centered copy of $\varrho^{(k_i+a)}(P)$ in $\varrho^{(k_{i+1})}(P)$ (given by Lemma \ref{lem:xx(P)}), coincide with $\PPP_{\sigma^{i+1}[1\ldots i]}^{(k_i)}$.} \label{fig:Positioning}
	\end{figure}

	It is left the verify the validity of properties \eqref{item:g_translated_copy_of..}, \eqref{item:g_nested+contain_0} and \eqref{item:g_distance_of_center_from_0}. By the induction hypothesis the origin is contained in $\PPP_{\sigma^{i+1}[1\ldots i]}^{(k_i)}$, hence properties \eqref{item:g_translated_copy_of..} and \eqref{item:g_nested+contain_0} follow directly from the construction. Note that the support of $\PPP_{\sigma^{i+1}[1\ldots i]}^{(k_i)}$ is indeed disjoint from the boundary of  $\supp{\PPP_{\sigma^{i+1}}^{(k_{i+1})}}$ by \eqref{item:prop_of_xx(P)} of Lemma \ref{lem:xx(P)}.
	To see \eqref{item:g_distance_of_center_from_0}, note that by our definition of the notion of a centered copy, 
	the point $\xx\left(\PPP_{\sigma^{i+1}}^{(k_{i+1})}\right)$ belongs to $\mathscr{T}_P$ (or to $\mathscr{T}_Q$, depends on $\sigma^{i+1}(i+1)$), which also contains the origin. Since for every $m\in\N$ the diameter of $\supp{\varrho^m(P)}$ is $\xi^m\diam{\supp{P}} = (\lambda_1^{1/d})^m\diam{\supp{P}}$ (see \S \ref{subsec:Notations}), and since $\mathscr{T}_P$ is a translate of $\varrho^{k_{i}+a}(P)$, 
	we have  
	\[\norm{\xx\left(\PPP_{\sigma^{i+1}}^{(k_{i+1})}\right)}\le \diam{\supp{\varrho^{k_{i}+a}(P)}} = 
	\lambda_1^{\frac{k_i+a}{d}}\diam{\supp{P}} = c_1\cdot \lambda_1^{k_i/d}.\] 
	In case $\sigma^{i+1}(i+1)=Q$, since $\diam{\supp{P}}=\diam{\supp{Q}}$, the same computation holds and the proof is complete.
\end{proof}

\begin{lem}\label{lem:g_TT_omega_def}
	For every infinite sequence $\omega\in\{P,Q\}^\N$ there exists a tiling $\TT_\omega\in\X_\varrho$ so that for every $i\in\N$ the tiling $\TT_\omega$ contains the patch $\PPP_{\omega[1\ldots i]}^{(k_{i})}$, defined in Proposition \ref{prop:g_main_step}.   
\end{lem}

\begin{proof}
	Let $\omega\in\{P,Q\}^\N$. By \eqref{item:g_nested+contain_0} of Proposition \ref{prop:g_main_step}, the sequence of patches $\left(\PPP_{\omega[1\ldots i]}^{(k_{i})}\right)_{i\in\N}$ is a nested sequence and by the proof of Proposition \ref{prop:g_main_step} it exhausts the plane. Thus 
	\[T_\omega := \bigcup_{i\in\N}\PPP_{\omega[1\ldots i]}^{(k_{i})}\]
	is a tiling of $\R^d$ and it satisfies the assertion. 
\end{proof}

Let $\Omega:=\{P,Q\}^\N$. Consider the equivalence relation on $\Omega$ in which $\omega\sim\omega'$ if the set $\{i\in\N \mid w(i)\neq w'(i)\}$ is finite. Since every equivalence class in this relation is countable, the cardinality of a set $\widetilde{\Omega}\subset \Omega$ of equivalence class representatives is $\cont$. We fix such a set of representatives $\widetilde{\Omega}$, then the following lemma completes the proof of Theorem \ref{thm:main_result_general}. 

\begin{lem}
	Let $\omega,\eta\in\widetilde{\Omega}$ be two distinct sequences, then the tilings $\TT_\omega$ and $\TT_\eta$, which are defined in Lemma \ref{lem:g_TT_omega_def}, are BD-non-equivalent.  
\end{lem}

\begin{proof}
	Since $\omega$ and $\eta$ are in $\widetilde{\Omega}$, and they are distinct, they differ at infinitely many places. Let $(i_m)_{m=1}^\infty$ be an increasing sequence so that $\omega(i_m)\neq\eta(i_m)$ for every $m$. We set $k_{i_m}$ as in \eqref{eq:h_and_k_i} and apply Corollary \ref{cor:from_Dirk+Yotam_for_tilings} with the sequence of sets $(A_m)_{m\in\N}$ defined by 
	\begin{equation*}
		B_{m}:=\supp{\varrho^{k_{i_m}}(P)_{\xx(P)}}, \qquad A_m:=\bigcup \left\{C(x)\in\mathcal{Q}_d \mid B_m\cap C(x)\neq\varnothing\right\} \quad (\text{see } \S \ref{subsec:BD}). 	
	\end{equation*}
	 
	By \eqref{item:g_distance_of_center_from_0} of Proposition \ref{prop:g_main_step} we have 
	\[\norm{\xx\left(\PPP_{\omega[1\ldots i_m]}^{(k_{i_m})} \right)},\norm{\xx\left(\PPP_{\eta[1\ldots i_m]}^{(k_{i_m})} \right)}\le c_1\cdot \lambda_1^{\frac{k_{i_m-1}}{d}}.\]
	Since $B_m$ and $\supp{\PPP_{\omega[1\ldots i_m]}^{(k_{i_m})}}$ differ by a translation and since $\xx(\varrho^{k_{i_m}}(P)_{\xx(P)})=0$ by definition, we deduce that 
	\[B_m\ \triangle\ \supp{\PP_{\omega[1\ldots i_m]}^{(k_{i_m})}}\subset \left\{\xx\in\R^d \mid \exists \yy\in\partial B_m, \ \norm{\xx - \yy}\le 
	c_1\cdot \lambda_1^{\frac{k_{i_m-1}}{d}} \right\},  \]
	and therefore 
	\[A_m\ \triangle\ \supp{\PP_{\omega[1\ldots i_m]}^{(k_{i_m})}}\subset \left\{\xx\in\R^d \mid \exists \yy\in\partial A_m, \ \norm{\xx - \yy}\le 
	\sqrt{d}\cdot c_1\cdot \lambda_1^{\frac{k_{i_m-1}}{d}} \right\} \df \mathcal{S}.  \]
	using e.g. \cite[Lammas 2.1 \& 2.2]{Laczk}, $\mu_d\left(\mathcal{S}\right)\le c(d)\cdot c_1^d\cdot \lambda_1^{k_{i_m-1}} \cdot \mu_{d-1}(\partial A_m)$ 
	and hence 
	\[\mu_d\left(A_m\ \triangle\ \supp{\PP_{\omega[1\ldots i_m]}^{(k_{i_m})}} \right)\le 
	c(d)\cdot c_1^d\cdot \lambda_1^{k_{i_m-1}} \cdot \mu_{d-1}(\partial A_m)
	,\]
	where $c(d)$ is a constant that depends on the dimension $d$. 
	Bounding the number of tiles in a region by the volume of the region divided by the smallest volume of a prototile, we obtain a constant $c_2>0$ that depends on $d$, $a$, $P$ and $\varrho$ so that 
	\begin{equation}\label{eq:g_tiles_in_Am_1}
	\absolute{\#[A_m]^{\TT_\omega} - \#\PP_{\omega[1\ldots i_m]}^{(k_{i_m})} } 
	\le \#\left[ A_m\ \triangle\ \supp{\PP_{\omega[1\ldots i_m]}^{(k_{i_m})}}\right]^{\TT_\omega}
	\le c_2 \cdot \lambda_1^{k_{i_m-1}} \cdot\mu_{d-1}(\partial A_m).
	\end{equation} 
	The above computations hold for $\PP_{\eta[1\ldots i_m]}^{(k_{i_m})}$ instead of $\PP_{\omega[1\ldots i_m]}^{(k_{i_m})}$ as well, and so we also have 
	\begin{equation}\label{eq:g_tiles_in_Am_2}
	\absolute{\#[A_m]^{\TT_\eta} - \#\PP_{\eta[1\ldots i_m]}^{(k_{i_m})} } \le c_2 \cdot \lambda_1^{k_{i_m-1}} \cdot \mu_{d-1}(\partial A_m).
	\end{equation}
	Combining \eqref{eq:g_tiles_in_Am_1} and  \eqref{eq:g_tiles_in_Am_2} we obtain that 
	
	\begin{equation}\label{eq:g_estimate_the_difference_on_Am}
	\absolute{\#[A_m]^{\TT_\omega} - \#[A_m]^{\TT_\eta} }\ge 
	\absolute{\#\PP_{\omega[1\ldots i_m]}^{(k_{i_m})} -\#\PP_{\eta[1\ldots i_m]}^{(k_{i_m})} } - 
	2c_2\cdot \lambda_1^{k_{i_m-1}} \cdot \mu_{d-1}(\partial A_m).
	\end{equation}
	Since $\omega(i_m)\neq\eta(i_m)$, and by Lemma \ref{lem:number_of_tiles_in_large_P-Q} and property \eqref{item:g_translated_copy_of..} of Proposition \ref{prop:g_main_step}, we have 
	
	\begin{equation}\label{eq:g_the_difference_on_special_P_^}
	\absolute{\#\PP_{\omega[1\ldots i_m]}^{(k_{i_m})} -\#\PP_{\eta[1\ldots i_m]}^{(k_{i_m})} } \ge c_0\absolute{\lambda_t}^{k_{i_m}}.  
	\end{equation}
	Relying on \eqref{eq:Hausdorff_scaling}, let $c_3>0$ be $\mathcal{H}_{d-1}(\partial \supp{P})$ times a constant that depends on $d$ such that 	
	\begin{equation}\label{eq:g_boundary_of_Am}
	\mu_{d-1}(\partial A_m) \le c_3 \left(\xi^{k_{i_m}}\right)^{d-1} 
	= c_3\left(\lambda_1^{(d-1)/d}\right)^{k_{i_m}},
	\end{equation}
	then by \eqref{eq:g_the_difference_on_special_P_^} and \ref{eq:g_boundary_of_Am} we have
	
	\begin{equation}\label{eq:g_with_the_boundary_of_Am}
	\absolute{\#\PP_{\omega[1\ldots i_m]}^{(k_{i_m})} -\#\PP_{\eta[1\ldots i_m]}^{(k_{i_m})} }/\mu_{d-1}(\partial A_m) \ge 
	\frac{c_0}{c_3} \left(\frac{\absolute{\lambda_t}}{\lambda_1^{(d-1)/d}}\right)^{k_{i_m}}.
	\end{equation}
	Note that $A_m\in\mathcal{Q}_d^*$, thus plugging \eqref{eq:g_estimate_the_difference_on_Am} 
	and \eqref{eq:g_with_the_boundary_of_Am} into \eqref{eq:non_BD_condition_tilings} we obtain that 
	
	\begin{equation}\label{eq:g_required-lower_bound}
	\frac{\absolute{\#[A_m]^{\TT_\omega} - \#[A_m]^{\TT_\eta} }}{\mu_{d-1}(\partial A_m)} \ge
	\frac{c_0}{c_3} \left(\frac{\absolute{\lambda_t}}{\lambda_1^{(d-1)/d}}\right)^{k_{i_m}} - 2c_2 \lambda_1^{k_{i_m-1}}.
	\end{equation}
	In view of \eqref{eq:h_and_k_i},  
	\[\left(\frac{\absolute{\lambda_t}}{\lambda_1^{(d-1)/d}}\right)^{k_{i_m}} = \left(\frac{\absolute{\lambda_t}}{\lambda_1^{(d-1)/d}}\right)^{h^{i_m-1}},
	\lambda_1^{k_{i_m-1}} = \left(\lambda_1^\frac{1}{h}\right)^{h^{i_m-1}} \text{and}\quad
	\lambda_1^\frac{1}{h} < \frac{\absolute{\lambda_t}}{\lambda_1^{(d-1)/d}},\] 
	which implies that the quantity on the right hand side of  \eqref{eq:g_required-lower_bound} tends to infinity with $m$. Then by Corollary \ref{cor:from_Dirk+Yotam_for_tilings}, the proof of the lemma and hence of Theorem \ref{thm:main_result_general} is complete. 	 
\end{proof}


\end{document}